\documentclass{amsart}
\usepackage{graphicx}
\usepackage{amsmath}
\usepackage{amsthm}
\usepackage{amssymb,bbm}%
\usepackage[numbers, square]{natbib}

\newcommand{\E}{\mathbb E}
\newcommand{\R}{\mathbb{R}}
\newcommand{\N}{\mathbb{N}}
\newcommand{\C}{\mathbb{C}}

\renewcommand{\P}{\mathbb{P}}

\newcommand{\DDD}{\mathbb{D}}

\newcommand{\Aa}{\mathcal{A}}
\newcommand{\Bb}{\mathcal{B}}

\newcommand{\MMM}{\mathbb{M}}

\newcommand{\eps}{\varepsilon}

\newcommand{\toprobab}{\overset{P}{\underset{n\to\infty}\longrightarrow}}

\newcommand{\toas}{\overset{a.s.}{\underset{n\to\infty}\longrightarrow}}

\renewcommand\Re{\operatorname{Re}}
\renewcommand\Im{\operatorname{Im}}

\theoremstyle{plain}
\newtheorem{theorem}{Theorem}[section]
\newtheorem{lemma}[theorem]{Lemma}
\newtheorem{corollary}[theorem]{Corollary}

\theoremstyle{definition}

\theoremstyle{remark}

\begin{document}

\author{Zakhar Kabluchko}
\address{Zakhar Kabluchko, Institute of Stochastics,
Ulm University,
Helmholtzstr.\ 18,
89069 Ulm, Germany}
\email{zakhar.kabluchko@uni-ulm.de}

\title[Critical points of random polynomials]{Critical points of random polynomials with independent identically distributed roots}

\begin{abstract}
Let $X_1,X_2,\ldots$ be independent identically distributed random variables with values in $\C$. Denote by $\mu$  the probability distribution of $X_1$. Consider a random polynomial $P_n(z)=(z-X_1)\ldots(z-X_n)$. We prove a conjecture of Pemantle and Rivin [arXiv:1109.5975] that the empirical measure $\mu_n:=\frac 1{n-1}\sum_{P_n'(z)=0} \delta_z$ counting the complex zeros of the derivative $P_n'$ converges in probability to $\mu$, as $n\to\infty$.
\end{abstract}
\maketitle

\section{Statement of the result}\label{sec:result}

A critical point of a polynomial $P$ is a root of its derivative $P'$. There are many results on the location of critical points of polynomials whose roots are known; see, e.g., \cite{rahman_schmeisser_book}.  One of the most famous examples is the Gauss--Lucas theorem stating that the complex critical points of any polynomial are located inside the convex hull of the complex zeros of this polynomial. \citet{pemantle_rivin} initiated the study of the probabilistic version of the problem. Let $X_1,X_2,\ldots$ be independent identically distributed (i.i.d.)\ random variables  with values in $\C$. Denote by $\mu$ the probability distribution of $X_1$. Consider a random polynomial
$$
P_n(z)=(z-X_1)\ldots(z-X_n). 
$$
Let $\mu_n$ be a probability measure which assigns to each critical point of $P_n$ the same weight, that is
$$
\mu_n = \frac 1{n-1}\sum_{z\in \C: P_n'(z)=0} \delta_z.
$$
We agree that the roots are always counted with multiplicities.  \citet{pemantle_rivin} conjectured that the distribution of roots of $P_n'$ should be stochastically close to the distribution of roots of $P_n$, for large $n$. In terms of logarithmic potentials, this means that the distribution of the equilibrium points of a two-dimensional electrostatic field generated by a large number of unit charges with i.i.d.\ locations should be close to the distribution of the charges themselves.  
\begin{theorem}\label{theo:main}
Let  $\mu$ be any probability measure on $\C$. Then, the sequence $\mu_n$ converges as $n\to\infty$ to $\mu$ in probability.
\end{theorem}
\citet{pemantle_rivin} proved Theorem~\ref{theo:main} for all measures $\mu$ having a finite $1$-energy. Later, \citet{subramanian} gave a proof if $\mu$ is concentrated on the unit circle. We refer to these two papers for more background information and motivation.
Our aim is to prove Theorem~\ref{theo:main} in full generality.

Few words about the mode of convergence. Let $\MMM$ be the set of probability measures on $\C$. Endowed with the topology of weak convergence, $\MMM$ becomes a Polish space.
We view $\mu_n$ as a random element with values in $\MMM$ and $\mu$ as a deterministic point in $\MMM$.
With this convention, Theorem~\ref{theo:main} states that for every open set $U\subset \MMM$ containing $\mu$,
$$
\lim_{n\to\infty} \P[\mu_n \notin U]=0.
$$
Since convergence in distribution and convergence in probability are equivalent if the limit is a.s.\ constant, see Lemma 3.7 in~\cite{kallenberg_book}, we can state Theorem~\ref{theo:main} as follows: the law of $\mu_n$ (viewed as a probability measure on $\MMM$) converges weakly to the unit point mass at $\mu$.

Our proof is based on the connection with the logarithmic potential theory (and does not follow the methods of~\cite{pemantle_rivin} and~\cite{subramanian}). The basic idea is to use the following formula (see, e.g., \S2.4.1 in~\cite{peres_book}): for every analytic function $f$ (which does not vanish identically),
\begin{equation}\label{eq:laplace_formula}
\frac 1 {2\pi} \Delta \log \left| f \right|= \sum_{z\in \C: f(z)=0} \delta_z.
\end{equation}
Here, $\Delta$ is the Laplace operator which should be understood in the distributional sense.
A similar method appeared in the study of roots of polynomials whose coefficients (not roots) are independent random variables, see~\cite{kabluchko_zaporozhets12a}, and in the random matrix theory; see~\cite{tao_vu}. We expect that there should numerous further applications of the method. On the heuristic level, we learned the idea to use formula~\eqref{eq:laplace_formula} from~\cite{derrida}; see also~\cite{kabluchko_klimovsky}.

\section{Proof}

\subsection{Method of proof}
Consider the logarithmic derivative of $P_n$:
\begin{equation}\label{eq:def_Ln}
L_n(z):= \frac{P_n'(z)}{P_n(z)}=\frac{1}{z-X_1}+\ldots+\frac{1}{z-X_n}.   
\end{equation}
The main steps of the proof of Theorem~\ref{theo:main} are collected in the following two lemmas.
\begin{lemma}\label{lem:Ln_toprobab_0}
There is a set $F\subset \C$ of Lebesgue measure $0$ such that for every $z\in \C\backslash F$ we have
\begin{equation}\label{eq:log_Ln_toprobab_0}
\frac 1n \log |L_n(z)|\toprobab 0.
\end{equation}
\end{lemma}
\begin{lemma}\label{lem:int_log_Ln_varphi_toprobab_0}
Let $\lambda$ be the Lebesgue measure on $\C$ and $\psi:\C\to\R$ any compactly supported continuous function. Then,
\begin{equation}\label{eq:lem_int_log_Ln}
\frac 1n \int_{\C} (\log |L_n(z)|) \psi(z) d\lambda(z) \toprobab 0.
\end{equation}
\end{lemma}
After the lemmas have been established, the proof of Theorem~\ref{theo:main} can be completed as follows.
It suffices to show that for every infinitely differentiable, compactly supported function $\varphi:\C\to\R$,
\begin{equation}\label{eq:main_varphi}
\frac 1n \sum_{z\in \C: P_n'(z)=0} \varphi(z)  \toprobab \int_{\C} \varphi d\mu.
\end{equation}
Indeed, \eqref{eq:main_varphi} implies that the law of $\mu_n$ converges weakly (as a probability measure on $\MMM$) to the unit point mass at $\mu$; see Theorem 14.16 in~\cite{kallenberg_book}. This implies convergence in probability since the limit is constant a.s.; see Lemma 3.7 in~\cite{kallenberg_book}.
%
To prove~\eqref{eq:main_varphi} we use the formula
\begin{equation}\label{eq:int_eq_delta_delta}
\frac 1{2\pi n} \int_{\C} (\log |L_n(z)|) \Delta \varphi(z) d\lambda(z) = \frac 1n \sum_{z\in \C: P_n'(z)=0} \varphi(z) - \frac 1n \sum_{z\in \C: P_n(z)=0} \varphi(z).
\end{equation}
It follows from~\eqref{eq:laplace_formula} with $f=P_n'$ and $f=P_n$ after subtraction and division by $n$.
As $n\to\infty$, the left-hand side of~\eqref{eq:int_eq_delta_delta} tends to $0$ in probability by Lemma~\ref{lem:int_log_Ln_varphi_toprobab_0}.
Since the zeros of $P_n$ are i.i.d.\ random variables, the second term in the right-hand side of~\eqref{eq:int_eq_delta_delta} tends to $\int \varphi d\mu$ in probability (and even a.s.)\ by the law of large numbers. This proves~\eqref{eq:main_varphi}. In the rest of the paper we are occupied with the proofs of Lemmas~\ref{lem:Ln_toprobab_0} and~\ref{lem:int_log_Ln_varphi_toprobab_0}. Let $\DDD_r(z)=\{x\in\C: |x-z| < r\}$ be the disk of radius $r>0$ centered at $z\in\C$ and $\bar \DDD_r(z)$ its closure. We also write $\DDD_r=\DDD_r(0)$ and $\bar \DDD_r=\bar \DDD_r(0)$.

\subsection{Proof of Lemma~\ref{lem:Ln_toprobab_0}}
First of all, let us stress that in general, \eqref{eq:log_Ln_toprobab_0} does not hold for every $z\in \C$ since it evidently fails if $z$ is an atom of $\mu$.  We need to introduce an exceptional set $F$. It consists of points at which $\mu$ has bad regularity properties.
Let
$$
\log_- z
=
\begin{cases}
|\log z|, &0\leq z\leq 1,\\
0,        &z\geq 1,
\end{cases}
\;\;\;\;\;
\log_+ z
=
\begin{cases}
0, &0\leq z\leq 1,\\
\log z, &z\geq 1.
\end{cases}
$$
Note that $\log_-0=+\infty$.
\begin{lemma}\label{lem:F_zero_set}
Let $F=\{z\in \C: \int_{\C}\log_-|y-z|d\mu(y)=\infty\}$. Then, the Lebesgue measure of $F$ is zero.
\end{lemma}
\begin{proof}
Recall that $\lambda$ is the Lebesgue measure on $\C$. We have, by Fubini's theorem,
$$
\int_{\C} \left(\int_{\C} \log_-|y-z| d\mu(y) \right) d\lambda(z)
=
\int_{\C} \left(\int_{\C} \log_-|z-y| d\lambda(z)\right) d\mu(y)
=
\frac {\pi}{2},
$$
where the second equality holds since the integral in the brackets is $\pi/2$ for every $y\in\C$ and $\mu$ is a probability measure. It follows that $\lambda(F)=0$.
\end{proof}

\begin{lemma}\label{lem:Ln_upper_bound}
For every $z\in \C\backslash F$ we have
$
\limsup_{n\to\infty} \frac 1n \log |L_n(z)| \leq 0 
$
a.s.
\end{lemma}
\begin{corollary}\label{lem:Ln_upper_bound_cor}
For every $z\in \C\backslash F$  and every $\eps>0$, we have
$$
\lim_{n\to\infty} \P\left[\frac 1n \log |L_n(z)|\geq \eps\right]=0.
$$
\end{corollary}
\begin{proof}[Proof of Lemma~\ref{lem:Ln_upper_bound}]
The idea is to show that the poles of $L_n$ do not approach $z$ too fast. Fix $\eps>0$. We have
$$
\int_{\C} \log_-|y-z| d\mu(y) \geq  \eps \sum_{n=1}^{\infty} \mu(\DDD_{e^{-\eps n}}(z)).
$$
Since the left-hand side is finite by the assumption $z\notin F$, the right-hand side must be finite, too. It follows that
$$
\sum_{n=1}^{\infty} \P\left[\frac 1 {|z-X_n|}>e^{\eps n}\right]
=
\sum_{n=1}^{\infty} \mu(\DDD_{e^{-\eps n}}(z))
<
\infty.
$$
By the Borel--Cantelli lemma, we have $\frac 1 {|z-X_n|}\leq e^{\eps n}$ for all but finitely many $n$. Also, $z$ is not an atom of $\mu$ (since $z\notin F$) and hence, $X_n\neq z$  for all $n\in\N$ a.s. It follows that there is an a.s.\ finite random variable $M$ such that
$$
|L_n(z)|\leq M+ ne^{\eps n}\leq e^{2\eps n},
$$
where the second inequality holds for large $n$. Thus, $\limsup_{n\to\infty}\frac 1n \log |L_n(z)|\leq 2\eps$. Since this holds for every $\eps>0$, the proof is completed.
\end{proof}

\begin{lemma}\label{lem:Ln_lower_bound}
For every $z\in \C$ and every $\eps>0$, we have
\begin{equation}\label{eq:lim_P_log_Ln_leq_0}
\lim_{n\to\infty} \P\left[\frac 1n \log |L_n(z)|\leq -\eps\right]=0.
\end{equation}
\end{lemma}
\begin{proof}
If $X_i=c$ a.s., then $L_n(z)=n/(z-c)$ and the lemma holds trivially. Assume therefore that the $X_i$'s are non-degenerate.
Given a real-valued random variable $\xi$ we denote by
$$
Q(\xi; \delta)=\sup_{t\in\R}\P[t\leq \xi\leq t+\delta],\;\;\; \delta>0,
$$
the concentration function of $\xi$.  We will use the fact that the concentration function of the sum of $n$ i.i.d.\ random variables decays like $O(1/\sqrt n)$. More precisely, by Theorem 2.22 on p.~76 in~\cite{petrov_book} there is an explicit absolute constant $C$ such that for every sequence of non-degenerate i.i.d.\ real-valued random variables $\xi_1,\xi_2,\ldots$ and for all $n\in\N$, $\delta>0$, we have
\begin{equation}\label{eq:concentration_ineq}
Q(\xi_1+\ldots+\xi_n;\delta)\leq C \frac{1+\delta}{\sqrt n}.
\end{equation}
Note that no moment requirements on the $\xi_i$'s are imposed. If $z\in\C$ is an atom of $\mu$, then~\eqref{eq:lim_P_log_Ln_leq_0} holds trivially since $|L_n(z)|=\infty$ with probability approaching $1$ as $n\to\infty$.  Fix $z\in\C$ which is not an atom of $\mu$. Consider the complex-valued random variables $Y_i=\frac 1 {z-X_i}$, $i\in\N$. Since we assume that the $X_i$'s are non-degenerate,  at least one of the random variables $\Re Y_1$ or $\Im Y_1$ is non-degenerate. Suppose for concreteness that $\Re Y_1$ is non-degenerate. Then,
$$
\P[|L_n(z)|\leq e^{-\eps n}]
\leq
\P\left[\left|\sum_{k=1}^{n} \Re Y_k\right|\leq e^{-\eps n}\right]
\leq
Q\left(\sum_{k=1}^{n} \Re Y_k , 2e^{-\eps n}\right)
\leq
\frac {2C}{\sqrt n}.
$$
The last inequality follows from~\eqref{eq:concentration_ineq} for $n$ large. This completes the proof.
\end{proof}
Combining Corollary~\ref{lem:Ln_upper_bound_cor} and Lemma~\ref{lem:Ln_lower_bound} we obtain Lemma~\ref{lem:Ln_toprobab_0}.


\subsection{Proof of Lemma~\ref{lem:int_log_Ln_varphi_toprobab_0}}
We already know from Lemma~\ref{lem:Ln_toprobab_0} that $\frac 1n \log |L_n(z)|$ converges to $0$ in probability for Lebesgue almost all $z\in \C$. To prove Lemma~\ref{lem:int_log_Ln_varphi_toprobab_0} we need to interchange the limit and the integral in~\eqref{eq:lem_int_log_Ln}. This is done by means of the following lemma whose proof can be found in~\cite{tao_vu}.
\begin{lemma}[Lemma~3.1 in~\cite{tao_vu}]\label{lem:tao_vu}
Let $(X,\Aa,\nu)$ be a finite measure space and $f_1,f_2,\ldots:X\to \R$  random functions which are defined over a probability space $(\Omega, \Bb, \P)$ and are jointly measurable with respect to $\Aa\otimes \Bb$.
Assume that:
\begin{enumerate}
\item For $\nu$-a.e.\ $x\in X$ we have $f_n(x)\to 0$ in probability, as $n\to\infty$.
\item For some $\delta>0$, the sequence $\int_X |f_n(x)|^{1+\delta} d\nu(x)$ is tight.
\end{enumerate}
Then, $\int_X f_n(x)d\nu(x)$ converges in probability to $0$.
\end{lemma}
Recall that $\psi$ is a continuous function with compact support. Let $r$ be such that the support of $\psi$ is contained in the disk $\DDD_r$. The first condition of Lemma~\ref{lem:tao_vu}, with $f_n(z)=\frac 1n (\log |L_n(z)|) \psi(z) $, $X=\DDD_r$, and $\nu=\lambda$ has been already verified in Lemma~\ref{lem:Ln_toprobab_0}. The second condition with $\delta=1$ follows from the next lemma.
\begin{lemma}\label{lem:tao_vu_tightness}
The sequence $\frac 1 {n^2} \int_{\DDD_r} \log^2|L_n(z)|d\lambda(z)$ is tight.
\end{lemma}
The rest of the paper is devoted to the proof of Lemma~\ref{lem:tao_vu_tightness}. First we need to prove a statement which is a uniform version of Lemma~\ref{lem:Ln_upper_bound}. This statement implies Lemma~\ref{lem:Ln_upper_bound}, but for clarity, we stated Lemma~\ref{lem:Ln_upper_bound} separately.  For $R>0$ define
\begin{equation}\label{eq:def_MnR}
M_n(R)=\sup_{|z|=R}|L_n(z)|.
\end{equation}
\begin{lemma}\label{lem:limsup_MnR}
There is a set $E\subset (0,\infty)$ of Lebesgue measure $0$ such that for every $R\in (0,\infty)\backslash E$ 
we have
\begin{equation}\label{eq:MnR_upper_bound}
\limsup_{n\to\infty} \frac 1n \log M_n(R) \leq 0\;\;\; \text{a.s.}
\end{equation}
\end{lemma}
\begin{proof}
The proof is similar to that of Lemmas~\ref{lem:F_zero_set} and~\ref{lem:Ln_upper_bound}. Let $\bar \mu$ be the radial part of $\mu$. This means that $\bar \mu$ is a measure on $[0,\infty)$ defined by $\bar \mu([0,s)) = \mu(|z|<s)$ for all $s>0$.
Define a set $E=\{R>0: \int_{\R} \log_-|x-R| d \bar \mu(x) = \infty\}$. By Fubini's theorem,
$$
\int_{\R} \left(\int_{\R} \log_-|x-R| d\bar\mu(x) \right) dR
=
\int_{\R} \left(\int_{\R} \log_-|R-x| dR\right) d\bar \mu(x)
=2
$$
since $\int_{\R} \log_-|R-x| dR = 2$ for every $x\in\R$ and $\bar \mu$ is a probability measure. Hence, $\lambda(E)=0$.
We now take $R\in (0,\infty)\backslash E$ and prove~\eqref{eq:MnR_upper_bound}. Fix $\eps>0$. We have
$$
\int_{\R} \log_-|x-R| d\bar \mu(x)
\geq
\eps \sum_{n=1}^{\infty} \bar \mu((R-e^{-\eps n}, R+e^{-\eps n}))
=
\eps \sum_{n=1}^{\infty}  \mu(\DDD_{R+e^{-\eps n}}\backslash  \bar\DDD_{R-e^{-\eps n}})
.
$$
The left-hand side is finite by the assumption $R\notin E$, hence the right-hand side is finite, too. It follows that
$$
\sum_{n=1}^{\infty} \P\left[\sup_{|z|=R} \frac 1 {|z-X_n|}>e^{\eps n}\right]
=
\sum_{n=1}^{\infty} \mu(\DDD_{R+e^{-\eps n}}\backslash  \bar \DDD_{R-e^{-\eps n}})
<
\infty.
$$
By the Borel--Cantelli lemma, we have $\sup_{|z|=R} \frac 1 {|z-X_n|}\leq e^{\eps n}$ for all but finitely many $n$. Note that $R$ is not an atom $\bar \mu$ (since $R\notin E$) and therefore, $|X_n|\neq R$ for all $n\in\N$ a.s. Hence, there is an a.s.\ finite random variable $M$ such that
$$
M_n(R)
\leq
\sum_{k=1}^n \sup_{|z|=R} \frac 1 {|z-X_k|}
\leq
M+ ne^{\eps n}\leq e^{2\eps n},
$$
for large $n$. In the first inequality we used~\eqref{eq:def_Ln} and~\eqref{eq:def_MnR}.
Since $\eps>0$ is arbitrary, the proof is completed.
\end{proof}

To prove Lemma~\ref{lem:tao_vu_tightness} we need to control the zeros and the poles of $L_n$ since at these points $\log |L_n(z)|$ becomes infinite.  We will use the Poisson--Jensen formula. Take some $R>r$. Denote by $x_{1,n},\ldots, x_{k_n,n}$ those zeros of $P_n$ which are located in the disk $\DDD_R$. They form a subset of $X_1,\ldots,X_n$. Let also $y_{1,n},\ldots, y_{l_n, n}$ be the zeros of $P_n'$ located in the disk $\DDD_R$. Note that $k_n\leq n$ and $l_n < n$. 
By the Poisson--Jensen formula, see~\cite[Chapter~8]{markushevich_book}, we have for any $z\in \DDD_R$ which is not a zero or pole of $L_n$,
\begin{equation}\label{eq:poisson_jensen}
\log |L_n(z)|
=
I_n(z;R)
+\sum_{l=1}^{l_n} \log \left|\frac{R(z-y_{l,n})}{R^2-\bar y_{l,n}z}\right|
-\sum_{k=1}^{k_n} \log \left|\frac{R(z-x_{k,n})}{R^2-\bar x_{k,n}z}\right|,
\end{equation}
where we use the notation
\begin{equation}\label{eq:poisson_integral}
I_n(z;R)=\frac 1 {2\pi} \int_{0}^{2\pi} \log |L_n(R e^{i\theta})|P_R(|z|, \theta-\arg z)d\theta
\end{equation}
and $P_R$ is the Poisson kernel
\begin{equation}\label{eq:poisson_kernel}
P_R(\rho, \varphi)=\frac{R^2-\rho^2}{R^2+\rho^2-2R\rho \cos \varphi},
\;\;\;
\rho\in[0,R],\, \varphi\in[0,2\pi].
\end{equation}

\begin{lemma}\label{lem:InzR_upper_bound}
There is $R>2r$ such that we have
$$
\limsup_{n\to\infty} \frac 1n \sup_{z\in\DDD_r} I_n(z;R) \leq 0\;\;\; \text{a.s.}
$$
\end{lemma}
\begin{corollary}\label{lem:InzR_upper_bound_cor}
There is $R>2r$ such that for every $\eps>0$,
$$
\lim_{n\to\infty} \P\left[\frac 1n \sup_{z\in\DDD_r} I_n(z;R) \geq \eps\right]=0.
$$
\end{corollary}
\begin{proof}[Proof of Lemma~\ref{lem:InzR_upper_bound}]
Choose any $R>2r$ not contained in the exceptional set $E$ of Lemma~\ref{lem:limsup_MnR}.
It follows from~\eqref{eq:poisson_kernel} that there is $C=C(r,R)$ such that $0< P_R(|z|, \theta) < C$ for all $z\in \DDD_r$ and  $\theta\in [0,2\pi]$.  It follows from~~\eqref{eq:def_MnR} and~\eqref{eq:poisson_integral} that $I_n(z;R)\leq C \log M_n(R)$ for all $z\in\DDD_r$. The proof is completed by using Lemma~\ref{lem:limsup_MnR}.
\end{proof}
In the sequel we choose $R\in (2r,\infty)\backslash E$ as in the above proof. In the next two lemmas we establish a lower bound for $I_n(z;R)$ which is uniform in $z\in \DDD_r$. First we consider the case $z=0$. Recall that $F$ is an exceptional set defined in Lemma~\ref{lem:F_zero_set}.
\begin{lemma}\label{lem:In0R_lower_bound}
Assume that $0\notin F$. There is a constant $A=A(R)$ such that
$$
\lim_{n\to\infty} \P\left[\frac 1n I_n(0;R)\leq  -A\right]=0.
$$
\end{lemma}
\begin{proof}
In the special case $z=0$ the Poisson--Jensen formula~\eqref{eq:poisson_jensen} takes the form
\begin{equation}\label{eq:poisson_jensen_0}
\frac 1n I_n(0;R)
=
\frac 1n \log |L_n(0)|
-\frac 1n \sum_{l=1}^{l_n} \log \left|\frac{y_{l,n}}{R}\right|
+\frac 1n \sum_{k=1}^{k_n} \log \left|\frac{x_{k,n}}{R}\right|.
\end{equation}
Recall that $x_{1,n},\ldots,x_{k_n,n}$ are those of the points $X_1,\ldots,X_n$ which belong to the disk $\DDD_R$. By the law of large numbers,
\begin{equation}\label{eq:LLN_for_log_zeros}
\frac 1n \sum_{k=1}^{k_n} \log \left|\frac{x_{k,n}}{R}\right| \toas -\E \left[\log_- \left|\frac{X_1}{R}\right|\right].
\end{equation}
The expectation on the right-hand side is finite. To see this note that $z\mapsto \log_-|z/R|-\log_-|z|$ is a bounded function with compact support and recall that $\E \log_-|X_1| <\infty$ by the assumption $0\notin F$. It follows from Lemma~\ref{lem:Ln_lower_bound} and~\eqref{eq:LLN_for_log_zeros} that there is $A_1=A_1(R)$ such that
$$
\lim_{n\to\infty}\P\left[\frac 1n \log |L_n(0)|\leq -1 \text{ or } \frac 1n \sum_{k=1}^{k_n} \log \left|\frac{x_{k,n}}{R}\right|\leq -A_1 \right]=0.
$$
For the second term on the right-hand side of~\eqref{eq:poisson_jensen_0} we have trivially
$$
\frac 1n \sum_{l=1}^{l_n} \log \left|\frac{y_{l,n}}{R}\right| \leq 0.
$$
The statement of the lemma follows with $A=A_1+1$.
\end{proof}
In the sequel we assume that $0\notin F$.  This is not a restriction of generality since in the case $0\in F$ we can choose any $a\notin F$ (which exists by Lemma~\ref{lem:F_zero_set}) and prove Theorem~\ref{theo:main} for the random variables $Y_i=X_i-a$ instead of $X_i$.
\begin{lemma}\label{lem:InzR_lower_bound}
There is a constant $B=B(r,R)$ such that
$$
\lim_{n\to\infty} \P\left[\frac 1n \inf_{z\in \DDD_r} I_n(z;R) \leq -B \right]=0.
$$
\end{lemma}
\begin{proof}
Write $q_n^+(\theta)=\frac 1n \log_+ |L_n(Re^{i\theta})|$ and $q_n^-(\theta)=\frac 1n \log_- |L_n(Re^{i\theta})|$, where $\theta\in[0,2\pi]$. Then, $q_n(\theta):=\frac 1n \log |L_n(Re^{i\theta})|=q_n^+(\theta)-q_n^-(\theta)$. Note that $q_n^+(\theta)\geq 0$ and $q_n^-(\theta)\geq 0$. By~\eqref{eq:poisson_kernel} and the assumption $R>2r$ there is a constant $C=C(r,R)>1$ such that $1/C< P_R(|z|, \theta) < C$ for all $z\in \DDD_r$, $\theta\in[0,2\pi]$. It follows that for $z\in\DDD_r$,
\begin{align*}
\frac {2\pi} n I_n(z;R)
&=
\int_0^{2\pi} q_n^+(\theta)P_R(|z|, \theta-\arg z)d\theta - \int_0^{2\pi} q_n^-(\theta)P_R(|z|, \theta-\arg z)d\theta\\
&\geq
\frac 1 C \int_0^{2\pi} q_n^+(\theta)d\theta -  C \int_0^{2\pi} q_n^-(\theta)d\theta\\
&=
\frac {2\pi C} n I_n(0;R)- \left(C-\frac 1C\right)  \int_0^{2\pi} q_n^+(\theta)d\theta\\
&\geq
\frac{2\pi C}{n} I_n(0;R)- 2\pi \left(C-\frac 1C\right) \frac 1n \log M_n(R).
\end{align*}
We have used that $I_n(0;R)=\frac{1}{2\pi}\int_{0}^{2\pi}\log |L_n(Re^{i\theta})|d\theta$. By Lemma~\ref{lem:In0R_lower_bound} and Lemma~\ref{lem:limsup_MnR} (recall that $R\notin E$) we have
$$
\lim_{n\to\infty} \P\left[ \frac 1n I_n(0;R) \leq -A \text{ or } \frac 1n \log M_n(R) > 1\right]=0.
$$
The statement of the lemma follows.
\end{proof}

We are in position to complete the proof of Lemma~\ref{lem:tao_vu_tightness}. Applying the inequality between the arithmetic and quadratic means several times  to the Poisson--Jensen formula~\eqref{eq:poisson_jensen} and dividing by $n^2$ we obtain
\begin{align}
\lefteqn{\frac 1{n^2}\log^2 |L_n(z)|}\label{eq:poisson_jensen_sq_ineq}\\
&\leq
 \frac{3}{n^2} I_n^2(z;R)
+\frac{3 l_n}{n^2}\sum_{l=1}^{l_n} \log^2 \left|\frac{R(z-y_{l,n})}{R^2-\bar y_{l,n}z}\right|
+\frac{3 k_n}{n^2} \sum_{k=1}^{k_n} \log^2 \left|\frac{R(z-x_{k,n})}{R^2-\bar x_{k,n}z}\right|.\notag
\end{align}
It follows from Corollary~\ref{lem:InzR_upper_bound_cor} and Lemma~\ref{lem:InzR_lower_bound} that the sequence $\frac{3}{n^2}\int_{\DDD_r} I_n^2(z;R)d\lambda(z)$ is tight. We estimate the remaining two terms in the right-hand side of~\eqref{eq:poisson_jensen_sq_ineq}.
We have, for some finite $C=C(r,R)$,
$$
\sup_{y\in \DDD_R} \int_{\DDD_r} \log^2 \left|\frac{R(z-y)}{R^2-\bar y z}\right|d\lambda(z) < C.
$$
To see this, note that  $|R^2-\bar y z|$ remains bounded below as long as $z\in\DDD_r$, $y\in\DDD_R$. and use the integrability of the squared logarithm.
Recall also that $k_n$ (resp., $l_n$) is the number of roots of $P_n$ (resp., $P_n'$) in the disk $\DDD_R$. Hence, both numbers do not exceed $n$. It follows that there is a deterministic constant $C_1=C_1(r,R)$ such that for every $n\in\N$,
$$
\frac{3 l_n}{n^2}\sum_{l=1}^{l_n} \int_{\DDD_r} \log^2 \left|\frac{R(z-y_{l,n})}{R^2-\bar y_{l,n}z}\right|d\lambda(z)
+\frac{3 k_n}{n^2} \sum_{k=1}^{k_n} \int_{\DDD_r} \log^2 \left|\frac{R(z-x_{k,n})}{R^2-\bar x_{k,n}z}\right|d\lambda(z)
\leq C.
$$
The sum of a tight sequence and an a.s.\ bounded sequence is tight. Hence, the sequence $\frac 1 {n^2} \int_{\DDD_r} \log^2 |L_n(z)|d\lambda(z)$ is tight. The proof of Lemma~\ref{lem:tao_vu_tightness} is complete.


\subsection*{Acknowledgment}
The author is grateful to D.\ Zaporozhets for numerous discussions on the topic of the paper.

\bibliographystyle{plainnat}
\bibliography{saddle_points_bib}
\end{document}